\documentclass[11pt]{amsart}

\usepackage[utf8]{inputenc}
\usepackage{amsmath,amsxtra,amssymb,latexsym,amscd,amsthm}

\theoremstyle{plain}

\usepackage{csquotes}

\usepackage{tikz-cd}
\usepackage{tkz-euclide}
\tikzset{every picture/.style={line width=0.25pt}} 
\usepackage[colorlinks]{hyperref} 
\hypersetup{
  pdftitle   = {Group operations on local moduli space of holomorphic vector bundles},
  pdfauthor  = {An-Khuong Doan},
  pdfcreator = {},
  linkcolor  = {blue},
  citecolor  = {blue},
  urlcolor   = {black}
}

\usepackage{url}

\theoremstyle{definition}

\theoremstyle{remark}

\numberwithin{equation}{section}

\theoremstyle{break}
\newtheorem{defi}{Definition}[section]
\newtheorem{thm}{Theorem}[section]
\newtheorem{prop}{Proposition}[section]
\newtheorem{lem}{Lemma}[section]
\newtheorem{coro}{Corollary}[section]

\theoremstyle{remark}
\newtheorem{rem}{Remark}[section]
\newtheorem*{ackn}{\textbf{Acknowledgements}}

\newtheorem*{thm*}{Theorem}

\newcommand{\Sp}{\operatorname{Spec}}

\newcommand{\Au}{\operatorname{Aut}}

\newcommand{\End}{\operatorname{End}}
\newcommand{\GL}{\operatorname{GL}}
\newcommand{\art}{\operatorname{\bf Art _{\mathbb{C}}}}
\newcommand{\grp}{\operatorname{\textbf{Grp}}}
\newcommand{\set}{\operatorname{\textbf{Set}}}
\newcommand{\defo}{\operatorname{Def}}
\begin{document}

\title{Group actions on local moduli space of holomorphic vector bundles}

\author{}
\address{}
\email{ }
\thanks{ }

\author{ An-Khuong DOAN}
\address{An-Khuong DOAN, IMJ-PRG, UMR 7586, Sorbonne Université,  Case 247, 4 place Jussieu, 75252 Paris Cedex 05, France}
\email{an-khuong.doan@imj-prg.fr }
\thanks{ }

\subjclass[2010]{14D15, 14B10, 32G}

\date{December 26, 2021.}

\dedicatory{ }

\keywords{Deformation theory, Moduli theory, Equivariance structure, Semi-universality}

\begin{abstract}
We prove that actions of complex reductive Lie groups on a holomorphic vector bundle over a complex compact manifold are locally extendable to its local moduli space. 
\end{abstract}

\maketitle
\tableofcontents
\section{Introduction}

 Since the foundational work of Kodaira-Spencer \cite{9} and Kuranishi \cite{10} on the existence of local moduli space parametrizing all the nearby structures of complex compact manifolds with respect to a given one, many similar existence theorems of such a kind have been proved for various other geometric objects among which we can mention: algebraic varieties \cite{14}, complex compact spaces \cite{6}, isolated singularities \cite{7}, holomorphic vector bundles \cite{12}, etc. This space and the associated family are usually known under the names: Kuranishi space and Kuranishi family, respectively, or sometimes semi-universal deformation for the latter. A natural question to pose is whether the automorphism group of the object under deformation could be lifted to an action on the Kuranishi space, which is compatible with the associated family. This is indeed the case when the associated family is universal. However, if the object in question has non-trivial automorphisms, which often happens in pratice, then the family cannot be universal in general.
 
 There are several attempts to answer this question, among which the work of D. S. Rim is outstanding \cite{13}. Namely, he gave an affirmative answer for a large class of local moduli problems (or equivalently for a large class of functors of artinian rings in Schlessinger's language). A vivid corollary of this beautiful work is the existence of equivariant structure on semi-universal deformation of projective schemes equipped with linearly reductive actions, unique up to non-canonical equivariant isomorphism. This is even more surprising that a counterexample constructed in \cite{2} or \cite{4} confirms the formal un-extendability for the non-reductive case in general. This is the reason why we shall focus only on reductive group actions. 
 
The main disadvantage of Rim's method is that his constructions are merely formal and that it only works well the algebraic world but not in the analytic world in which we have to deal with deformations problems associated to analytic objects where the convergence is naturally needed. This is the principal difference between the two worlds. Therefore, the group extension problem seems even harder on the analytic side. In some specific situations, we can expect to prove the convergence. We can mention the case where the object under deformation is a complex compact manifold equipped with actions of a reductive complex Lie group, of which the main ingredient of the proof is a combination of an equivariant version of Kuranishi's classical construction of local moduli spaces of complex compact manifolds and representations of reductive complex Lie group (cf. \cite{3}). Inspired by this result, we continue to consider the case where the analytic object under deformation is a holomorphic vector bundle on which a reductive complex Lie group acts holomorphically. It should be noted as well that our main result in this paper is the existence of group operations on local moduli space for reductive subgroups of the automorphism group of the considered bundle without any further assumption on the bundle (cf. Theorem \ref{t4.1} and Corollary \ref{c4.1}) while N. Buchdahl, G. Schumacher in \cite{1} proved that the whole automorphism group can be lifted to a compatible group action on its local moduli space provided that the Kählerianness on the given complex manifold and the poly-stability assumption on the holomorphic vector bundle under consideration are added (see \cite[Theorem 5]{1}).  This suggests once again that the group extension problem is not feasible in general unless some additional hypothesis is imposed on either the group or on the considered geometric structure. It would be of great interest if one could find a counterexample in either case.
 
Let us now outline the organization of this article. First, we give a general description of holomorphic vector bundles and their deformations in $\S$\ref{s2} and $\S$\ref{s3}, respectively. Next, we prove the existence of reductive group actions on Kuranishi spaces of vector bundle in $\S$\ref{s4}. The main techniques are essentially inspired from those in \cite{3}. In $\S$\ref{s5}, we compute the differential graded Lie algebra associated to the deformation problem of holomorphic vector bundles, from which a formal version of reductive group actions on local moduli space of holomorphic vector bundles, obtained in $\S$\ref{s4}, follows easily. At last, in $\S$\ref{s6}, we introduce a general philosophy hidden behind our work.

 \begin{ackn} We would like to warmly thank Prof. S. Kosarew for many useful discussions and references. We show gratitude to Prof. Julien Grivaux for reading the first version of this note and giving several precious comments. Finally, we are specially thankful to the referee whose dedicated work led to a great amelioration of this paper.
 \end{ackn}
 \section{Holomorphic vector bundles}\label{s2}
 Let $E$ be a differentiable  vector bundle of rank $r$ over a compact complex compact manifold $X$. Let $A^{p,q}(E)$ be the  space of $(p,q)$-forms with values in $E$.
 \begin{defi} A semi-connection on $E$ is a $\mathbb{C}$-linear map $D:\;A^{0,0}(E) \rightarrow A^{0,1}(E) $ satisfying the Leibnitz rule, i.e. $$D(fs)=(\overline{\partial} f)s+f\cdot Ds$$
 for $f \in C^{\infty}(X)$ and $s\in A^{0,0}(E) $.
 \end{defi}
 It is evident that each semi-connection $D$ can be extended to a first order differential operator $D:\;A^{p,q}(E) \rightarrow A^{p,q+1}(E) $. Moreover, if we let $\mathcal{D}(E)$ be the space of semi-connections on $E$, then it is well-known that for a fixed semi-connection $D_0$,  $\mathcal{D}$ can be identified with  $D_0+ A^{0,1}(\End(E))$ where $\End(E)$ is the group of differentiable endomorphisms of $E$ (inducing the identity on the base manifold) and thus an affine space. 
 \begin{defi} A semi-connection $D$ is called a holomorphic structure if 
 $$D\circ D=0.$$
 This condition is called the integrability condition.
 \end{defi}
 Now, let $\mathcal{H}(E)$ be the subset of $\mathcal{D}(E)$, consisting of holomorphic semi-connections $D$. Then $\mathcal{H}(E)$ is nothing but the set of holomorphic bundle structure on the differentiable complex vector bundle $E$.
 If we denote the group of differentiable automorphisms of $E$ (inducing the identity on the base manifold $X$) by $\GL(E)$ then $\End(E)$ be thought of as the Lie algebra of $\GL(E)$ and an action of $\GL(E)$ on $\mathcal{D}(E)$ is given by $$g.D=g^{-1}\circ D \circ g$$ where $g \in \Au(E)$ and $D\in \mathcal{D}(E)$.
 \begin{rem} If $E$ is a holomorphic vector bundle whose holomorphic structure is uniquely determined by a holomorphic connection $D$ then $g.D=D$ for any holomorphic automorphism of $E$. 
 \end{rem}
 \begin{prop}\label{p2.1} Let $D$ be a holomorphic connection on $E$ and $\alpha$, $\alpha_1$, $\alpha_2 \in A^{0,1}(\End(E))$.
 \begin{enumerate} 
\item[(i)]  $D+\alpha$ defines a structure of holomorphic vector bundle on $E$ if and only if 
 $$\mathcal{P}_{D}(\alpha):= D\alpha +\alpha \wedge \alpha=0$$ where the wedge product $\alpha \wedge \alpha$ is given by the usual wedge product in the form part and the usual composition of endomorphisms in $\End(E)$.
 \item[(ii)] Let $g \in \GL(E)$. Then $g.(D+\alpha_2)=D+\alpha_1$ if and only if $$g\circ \alpha_1 -\alpha_2 \circ g -Dg=0.$$
 \end{enumerate}  
 \end{prop}

 \section{Deformation of holomorphic vector bundles}\label{s3}
We first recall some basic definitions in deformation theory of holomorphic vector bundles (cf. \cite{12} for more details). Let $\mathfrak{B}$ be the category of germs of pointed complex spaces $(B,0)$ (a complex space with a reference point) whose associated reduced complex space is a point and  $X$ a complex compact manifold. Let $E$ be a holomorphic vector bundle over $X$, whose associated holomorphic semi-connection and underlying differentiable complex vector bundle will be denoted by $D_E$ and $\underline{E}$, respectively. We fix once for all a sufficiently large integer $k$. Consider the Hilbert space $ A^{0,1}(\End(E))_k$ (resp. $ A^{0,2}(\End(E))_{k-1}$)  obtained by completing $ A^{0,1}(\End(E)))$ (resp. $A^{0,2}(\End(E))))$  with respect to the Sobolev $k$-norm (resp $k-1$-norm). We have an induced analytic map coming from Proposition \ref{p2.1} 
\begin{align*}
\mathcal{P}_{D_E} :A^{0,1}(\End(E))_k &\rightarrow A^{0,2}(\End(E))_{k-1}\\
\alpha &\mapsto D_E.\alpha +\alpha \wedge \alpha
\end{align*} which then gives a germ of a Banach analytic space $(\mathcal{P}_{D_E}^{-1}(0), 0)$.
\begin{defi} 
\begin{enumerate} \label{d3.1}
\item[(i)] \label{d3.1i} A local deformation of $E$ over a germ of pointed complex spaces $(B,0)$ is a pair $(\pi,(B,0) )$ where $\pi$ a holomorphic map $\pi$ from $(B,0)$ to the germ of Banach analytic spaces $(\mathcal{P}_{D_E}^{-1}(0), 0)$ which is of class $C^{\infty}$ on $X\times D$ where $D$ is the ambient space of $(B,0)$.
\item[(ii)]\label{d3.1i} Two local deformations $(\pi,(B,0) )$ and $(\sigma,(B,0) )$ of $E$ are equivalent if there exists an analytic map $$\rho:\;(B,0 ) \rightarrow (\GL(E)_{k+1},\mathrm{Id}_E)$$ which is of class $C^{\infty}$ such that
$$\rho(t)\circ \pi(t) -\pi(t) \circ \rho(t) -D\rho(t)=0 $$
in $A^{0,1}(\End(E))$.
\end{enumerate}
\end{defi}
\begin{rem}In other words, if  $(\pi,(B,0) )$ a local deformation of $E$, the  we obtain a family of holomorphic vector bundles $\lbrace E_{\pi(b)} \rbrace_{b \in B}$ varying holomorphically in $b$. Here, $E_{\pi(b)}$ is the differentiable complex vector bundle $\underline{E}$ equipped with the holomorphic structure $D_{E}+\pi(b)$. In addition, it defines a holomorphic vector bundle $\mathcal{E}_{B} \rightarrow  X\times B$ such that the restriction on $X\times \lbrace 0\rbrace$ is nothing but the initial holomorphic vector bundle $E$. We call $\mathcal{E}_{B} \rightarrow  X\times B$ the associated bundle to the local deformation $(\pi,(B,0) )$.

\end{rem}
If $(\pi, (B,0))$ is a local deformation of $E$ and $f:\;(S,0)\rightarrow (B,0)$ is an analytic map of germs of complex spaces then the pullback of $(\pi, (B,0))$ by $f$ is defined to be the local deformation $(\pi\circ f,(S,0))$, which we shall denote by $f^*(\pi,(B,0))$.

\begin{defi} A local deformation $(\pi,(B,0))$ of $E$ is semi-universal if any other local deformation $(\rho,(S,0))$ of $E$ is defined by the pullback of $(\pi,(B,0))$ under some holomorphic map from $(T,0)$ to $(S,0)$, whose differential at the reference point is unique.    
\end{defi}
The following fundamental theorem is essentially due M. Kuranishi (\cite[Theorem 1]{12}).
\begin{thm} Let $E$ be a holomorphic vector bundle defined over a compact complex manifold $X$. Then there exists a semi-universal local deformation of $E$, unique up to non-canonical isomorphisms. 
\end{thm}

Next, let us take a moment to recall the definition of group actions on complex spaces. For the sake of completeness, we recall first that a mapping $\alpha$ from a real analytic (resp. complex) manifold $W$ to a Fréchet space $F$ over $\mathbb{C}$ is called \textit{real analytic} (reps. \textit{holomorphic}) if for each point $w_0\in W$ there exists an open coordinate neighborhood $N_{w_0}$ and a real analytic (resp. holomorphic) coordinate system $t_1,\ldots,t_n $ in $N$ such that $t_i(w_0)=0$ and for all $w\in N$, we have
that $$\alpha(w)=\sum a_{i_1,\ldots,i_n}t_1^{i_1}(w)\ldots t_n^{i_n}(w) $$ where $a_{i_1,\ldots,i_n} \in F$ and the convergence is absolute with respect to any continuous semi-norm on $F$. Furthermore, by a $C^p$-map, we insinuate a $p$-times continuously differentiable function. Let $G$ be a real (resp. complex) Lie group and $X$ a complex space. A $G$-action on $X$ is given by a group homomorphism $\Phi:\; G \rightarrow \Au(X)$, where $\Au(X)$ is the group of biholomorphisms of $X$. 
\begin{defi} The $G$-action determined by $\Phi$ is said to be real analytic (resp. holomorphic) if for each open relatively compact $U \Subset X$ and for each open $V\subset X$, the following conditions are satisfied
\begin{enumerate}
\item[(i)]$W:=W_{\overline{U},V}:=\lbrace g\in G \mid g\cdot \overline{U}\subset V \rbrace $ is open in $G$,
\item[(ii)]the map \begin{align*}
*:W&\rightarrow \mathcal{O}(U)\\
g &\mapsto f\circ g\mid_U
\end{align*} is real analytic (resp. holomorphic) for all $f\in \mathcal{O}(V)$ ,
\end{enumerate}
where $\overline{U}$ is the closure of $U$ and $\mathcal{O}(P)$ is the set of holomorphic functions on $P$ for any open subset $P$ of $X$ ($\mathcal{O}(P)$ is equipped with the canonical Fréchet topology).
\end{defi}
Finally, it is time for us to introduce $G$-equivariant deformations, which are of central interest of the article. As before, let $X$ be a complex compact manifold over which a holomorphic vector bundle $E$ is defined. Let $G$ be a subgroup of the group of holomorphic automorphisms of $E$.
\begin{defi} \label{d3.4}
 A real analytic (resp.\ holomorphic) $G$-equivariant local deformation of $E$ is a usual local deformation of  $(\pi,(B,0))$ of $E$ whose associated bundle $\mathcal{E}_B$ can be equipped with a real analytic (resp.\ holomorphic) $G$-action  extending the given (resp.\  holomorphic) $G$-action on $E$ and a real analytic (resp. holomorphic) $G$-action on $B$ in a way that the projection $\mathcal{E}_{B} \rightarrow  X\times B$ is a $G$-equivariant map with respect to these actions. We call these extended actions a real analytic (resp.\  holomorphic) $G$-equivariant structure on $\mathcal{E}_{B} \rightarrow  X\times B$.
\end{defi}
\begin{rem} We make the following convention. Whenever we have a $G$-action on $(B,0)$,  the $G$-action on $X\times B$ in Definition \ref{d3.4} is defined by $$g(x,b)=(x,g.b)$$ for $g \in G$ and $(x,b)\in X\times B$. This is exactly the action with respect to which we want the projection $\mathcal{E}_{B} \rightarrow  X\times B$ to be $G$-equivariant.  Moreover, the restriction of the $G$-action on $\mathcal{E}_B$ on the $\mathcal{E}_{B,0}$ is nothing but the initial $G$-action on $E$.
\end{rem}
\hfill \break
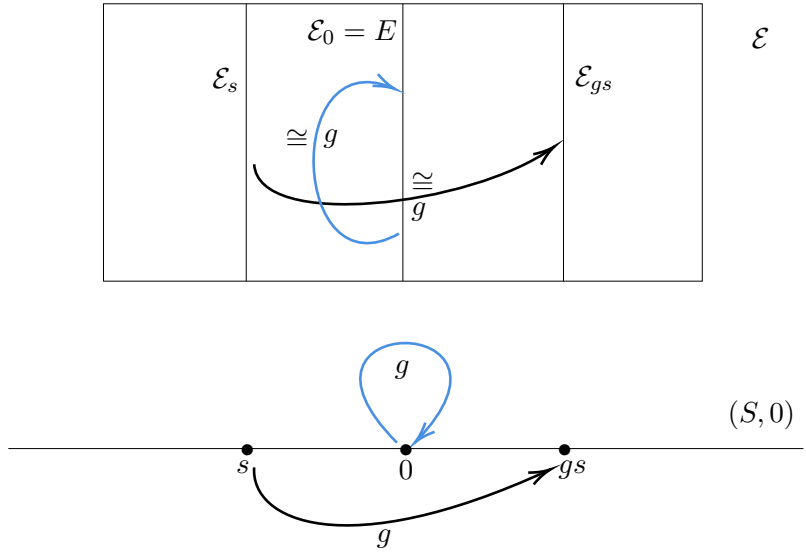
\begin{center}
\begin{tikzpicture}[x=0.75pt,y=0.75pt,yscale=-1,xscale=1]

\draw    (71,235.5) -- (269,235.5) -- (472,235.5) ;
\draw   (119,11) -- (421,11) -- (421,151) -- (119,151) -- cycle ;
\draw    (351,11) -- (351,151) ;
\draw    (270,11) -- (270,151) ;
\draw    (191,11) -- (191,151) ;
\draw  [line width=1]  (195,92) .. controls (196.98,125.66) and (304.32,113.26) .. (345.28,83.89) ;
\draw [shift={(346.5,83)}, rotate = 503.13] [color={rgb, 255:red, 0; green, 0; blue, 0 }  ][line width=1]    (10.93,-3.29) .. controls (6.95,-1.4) and (3.31,-0.3) .. (0,0) .. controls (3.31,0.3) and (6.95,1.4) .. (10.93,3.29)   ;
\draw  [line width=1]  (195,245) .. controls (193.01,271.89) and (240.52,293.78) .. (342.95,246.71) ;
\draw [shift={(344.5,246)}, rotate = 515.12] [color={rgb, 255:red, 0; green, 0; blue, 0 }  ][line width=1]    (10.93,-3.29) .. controls (6.95,-1.4) and (3.31,-0.3) .. (0,0) .. controls (3.31,0.3) and (6.95,1.4) .. (10.93,3.29)   ;
\draw[line width=1] [color={rgb, 255:red, 74; green, 144; blue, 226 }  ,draw opacity=1 ]   (268,127) .. controls (213.55,157.69) and (209.09,30.58) .. (264.31,53.27) ;
\draw [shift={(266,54)}, rotate = 204.52] [color={rgb, 255:red, 74; green, 144; blue, 226 }  ,draw opacity=1 ][line width=1]    (10.93,-3.29) .. controls (6.95,-1.4) and (3.31,-0.3) .. (0,0) .. controls (3.31,0.3) and (6.95,1.4) .. (10.93,3.29)   ;

\draw [line width=1][color={rgb, 255:red, 74; green, 144; blue, 226 }  ,draw opacity=1 ]   (267,232.5) .. controls (200.33,165.83) and (340.59,165.5) .. (276.97,231.5) ;
\draw [shift={(276,232.5)}, rotate = 314.57] [color={rgb, 255:red, 74; green, 144; blue, 226 }  ,draw opacity=1 ][line width=1]    (10.93,-3.29) .. controls (6.95,-1.4) and (3.31,-0.3) .. (0,0) .. controls (3.31,0.3) and (6.95,1.4) .. (10.93,3.29)   ;

\draw (172.75,42.15) node [anchor=north west][inner sep=0.75pt]    {$\mathcal{E}_{s}$};
\draw (220.75,18.15) node [anchor=north west][inner sep=0.75pt]    {$\mathcal{E}_{0} =E$};
\draw (355.75,41.15) node [anchor=north west][inner sep=0.75pt]    {$\mathcal{E}_{gs}$};
\draw (444.75,22.15) node [anchor=north west][inner sep=0.75pt]    {$\mathcal{E}$};
\draw (184.75,240.15) node [anchor=north west][inner sep=0.75pt]    {$s$};
\draw (432.75,210.15) node [anchor=north west][inner sep=0.75pt]    {$( S,0)$};
\draw (266.75,240.15) node [anchor=north west][inner sep=0.75pt]    {$0$};
\draw (273.25,95.15) node [anchor=north west][inner sep=0.75pt]    {$\cong $};
\draw (273.25,110.15) node [anchor=north west][inner sep=0.75pt]    {$g$};
\draw (186.75,232.15) node [anchor=north west][inner sep=0.75pt]    {$\bullet $};
\draw (346.75,232.15) node [anchor=north west][inner sep=0.75pt]    {$\bullet  $};
\draw (266.75,232.15) node [anchor=north west][inner sep=0.75pt]    {$\bullet  $};
\draw (347.25,240.15) node [anchor=north west][inner sep=0.75pt]    {$gs$};
\draw (228.75,73.15) node [anchor=north west][inner sep=0.75pt]    {$g$};
\draw (209.75,73.15) node [anchor=north west][inner sep=0.75pt]    {$\cong $};
\draw (255.75,275.15) node [anchor=north west][inner sep=0.75pt]    {$g$};

\draw (264.75,190.15) node [anchor=north west][inner sep=0.75pt]    {$g$};
\end{tikzpicture}
\end{center}

As a matter of course, the main goal of this paper is to construct a real analytic (resp.\  holomorphic) $G$-equivariant semi-universal local deformation of a complex vector bundle with a real analytic (resp.\  holomorphic) $G$-action. Intuitively, the expected extended $G$-action on the ``Kuranishi space"  permutes the nearby holomorphic structures and keeps the central one untouched.
\begin{rem}
For simplicity, by $G$-actions (resp.\  $G$-equivariant local deformations), we really mean real analytic $G$-actions (resp.\  real analytic $G$-equivariant local deformations).
\end{rem}

\section{Existence of group operations on local moduli spaces}\label{s4}
In this section, we shall follow strictly the construction given in \cite{12}, in which the $G$-action can be naturally integrated along the lines. As usual, let $X$ be a compact complex manifold over which a holomorphic vector bundle $E$ is defined and $G$ a subgroup of the group of holomorphic automorphisms of $E$. The case that $G$ is a compact Lie group shall be treated first. It should be noted that $G$ will induce a holomorphic $G$-action on the bundle $\End(E)$ given by 
\begin{equation}\label{e4.1} g.\sigma =g^{-1}\circ \sigma \circ g 
\end{equation}  for $g \in G$ and $\sigma \in \End(E)$. Consequently, we obtain natural $G$-action on $A^{0,1}(\End(E))$, $A^{0,2}(\End(E))$ and then on $H^{1}(X,\End(E))$. The compactness of $G$ permits us to impose a $G$-invariant Hermitian metric on $\End(E)$, by the unitary trick. This metric induces a $G$-invariant metric on $A^{0,1}(\End(E))$ with respect to which a formal adjoint  $$D_{E}^*: A^{p,q}(\End(E)) \rightarrow A^{p,q-1}(\End(E)) $$ of $$D_{E}:A^{p,q-1}(\End(E)) \rightarrow A^{p,q}(\End(E))$$ is provided. Since the $G$-action is holomorphic then the connection $D_E$ is a $G$-equivariant differential operator. The equivariance of $D_{E}^*$ follows from the one of $D_E$ and from the $G$-invariance of the imposed metric. As a matter of fact, the Laplace-Beltrami operator associated to $\End(E)$, $$\square_{E}:= D_{E}^* D_{E}+D_{E}D_{E}^*$$ is $G$-equivariant as well. Moreover, the principal part of $\square_{E}$ coincides with that of the usual Laplace-Beltrami operator. The latter is well-known to be a strongly elliptic self-adjoint operator of second order. Hence, so is $\square_{E}$. This is where we can make use of the Hodge theory for the bundle $\End(E)$. Namely, if we denote the space of harmonic $(0,1)$-form with coefficients in $\End(E)$ by $\mathcal{H}^{0,1}$ then $\mathcal{H}^{0,1}$ can be naturally identified with the first cohomology $H^1(X,\End(E))$ and the following orthogonal decomposition is available
\begin{align*}
 A^{0,1}(\End(E)) &=\mathcal{H}^{0,1}\oplus \square_{E} A^{0,1}(\End(E)) \\ 
 &= \mathcal{H}^{0,1}\oplus D_E A^{0,0}(\End(E))\oplus D_E^* A^{0,2}(\End(E))
\end{align*} together with two linear operators:
\begin{enumerate}
\item[(i)] The Green operator: $\mathcal{G}:\; A^{0,1}(\End(E)) \rightarrow \square_{E} A^{0,1}(\End(E))$,
\item[(ii)] The harmonic projection: $\mathrm{P}_{\mathcal{H}^{0,1}}:\; A^{0,1}(\End(E)) \rightarrow \mathcal{H}^{0,1}$
\end{enumerate} such that \begin{equation} \label{e4.2} \mathrm{Id}_{A^{0,1}(\End(E))}= \mathrm{P}_{\mathcal{H}^{0,1}}+\square_{E}\mathcal{G}.
\end{equation}
\begin{lem} The operators $\mathrm{P}_{\mathcal{H}^{0,1}}$ and $\mathcal{G}$ are $G$-equivariant.
\end{lem}
Consider the map 
\begin{align*}
\mathcal{P}_{D_E} :A^{0,1}(\End(E))_k &\rightarrow A^{0,2}(\End(E))_{k-1}\\
\alpha &\mapsto D_E.\alpha +\alpha \wedge \alpha
\end{align*} defined the previous section.

\begin{lem} $\mathcal{P}_{D_E}$ is $G$-equivariant.

\end{lem}
\begin{proof} Indeed, it suffices to prove that for $g\in G$ and $\alpha \in A^{0,1}(\End(E))$, we have
$$g.(\alpha\wedge \alpha)=g.\alpha \wedge g.\alpha.$$
However, this follows immediately from the fact that $G$ acts trivially on the form part and acts on the endomorphism part by the rule (\ref{e4.1}).
\end{proof}
Now, we are ready to state our main result.
\begin{thm} \label{t4.1} Let $X$ be a compact complex manifold over which a holomorphic vector bundle $E$ is defined. Let $G$ be a compact real Lie subgroup of the automorphism group of $E$. Then there exists a real analytic $G$-equivariant semi-universal local deformation of $E$. 
\end{thm}
\begin{proof} We consider the following set
$$Q:= \lbrace \alpha \in A^{0,1}(\End(E))_k\vert \left \| \alpha \right \|_k < \epsilon,\; D_E^*\alpha=0,\; D_E^*\circ \mathcal{P}_{D_E}(\alpha)=0 \rbrace  .$$ Each $\alpha \in Q$ is a solution of the elliptic partial differential equation
\begin{equation}\square_{E}\alpha + D_E^*(\alpha\wedge \alpha)=0.
\end{equation} \label{e4.3}Thus, $Q$ is actually a subset of $A^{0,1}(\End(E))$. We claim further that $Q$ is a direct finite-dimensional sub-manifold of an open neighborhood of $0$ in $ A^{0,1}(\End(E))_k$. Indeed, let us  take into account the following auxiliary analytic function
\begin{align*}
\gamma :A^{0,1}(\End(E))_k &\rightarrow \mathcal{H}^{0,1}\oplus D_E^*  A^{0,1}(\End(E))_{k} \oplus D_E^*  A^{0,2}(\End(E))_{k-1}\\
\alpha &\mapsto  \left( \mathrm{P}_{\mathcal{H}^{0,1}}(\alpha),D_E^*\alpha,D_E^*\circ \mathcal{P}_{D_E}(\alpha) \right)
\end{align*} The differential of $\gamma$ at $0$ is $$d\gamma_0 =(\mathrm{P}_{\mathcal{H}^{0,1}},D_E^*,D_E^*D_E)$$ whose inverse can be explicitly given by

$$(d\gamma_0)^{-1}(\alpha_0,\alpha_1,\alpha_2)=\alpha_0+\mathcal{G}D_E(\alpha_2)+\mathcal{G}(\alpha_2)$$ so that $\gamma$ is a local analytic isomorphism around $0$ due to the inverse mapping theorem for Banach manifolds. Therefore, locally, $$Q=\gamma^{-1}(\mathcal{H}^{0,1}\times 0 \times 0)$$ where $\mathcal{H}^{0,1}$ is known to be a finite dimensional vector space over $\mathbb{C}$. This justifies the claim for $\epsilon$ small enough. Now, on one hand, note that because each component of $\gamma$ is $G$-equivariant then so is $\gamma$. On the other hand, $\mathcal{H}^{0,1}\times 0 \times 0$ is $G$-invariant. Thus, after shrinking $Q$ (if necessary), we obtain a $G$-action on $Q$. For the sake of $G$-equivariance of $\mathcal{P}_{D_E}$, the germ of Banach analytic spaces $(\mathcal{P}_{D_E}^{-1}(0), 0)$ carries as well a $G$-action. Let $T =Q \cap \mathcal{P}_{D_E}^{-1}(0)$. Then the germ of complex space $(T,0)$ and the inclusion 
\begin{equation}\label{e4.4}\omega:\; (T,0) \rightarrow (\mathcal{P}_{D_E}^{-1}(0),0)
\end{equation} will determine a semi-universal local deformation $\mathcal{E} \rightarrow X\times (T,0)$ of $E$. This can be carried out in a similar way as in the rest of the proof of \cite[Theorem 1]{12}. What is new here is the fact that the analytic map $\omega$ is further $G$-equivariant.

Now, we would like to equip $\mathcal{E}$ with a compatible $G$-action so that the local deformation $\mathcal{E} \rightarrow X\times (T,0)$ becomes $G$-equivariant in the sense of Definition \ref{d3.4}. First, as a complex space, the bundle $\mathcal{E}$ is nothing but $\underline{E} \times T$ equipped with the complex structure induced by the holomorphic semi-connection $\omega$. In addition, each fiber $\mathcal{E}_{s}$ is exactly $\underline{E}$ equipped with a structure $\omega(s)$ of holomorphic vector bundles and in particular a structure $J_s$ of complex manifolds on the differentiable manifold $\underline{E}$. Thus, the $G$-equivariance of $\omega$ implies that \begin{equation} \label{e4.5} dg.J_s =J_{gs}.dg \end{equation} where $g\in G$ and $dg$ is the differential of $g$. At this point, it should be noted that $g$ is a biholomorphism with respect to the complex structure $J_0$ of the initial holomorphic bundle $E$ and that elsewhere we think of $g$ just as a diffeomorphism of $\underline{E}$. These discussions allow us to define the following $G$-action on $\underline{E} \times T$,
\begin{align*}
g :\underline{E} \times T &\rightarrow \underline{E} \times T\\
(e,t) &\mapsto (ge,gt)
\end{align*}
for each $g \in G$. We claim that the diffeomorphism $g$ defined in this way is actually a biholomorphism of $\mathcal{E}$. It is the same as verifying that the differential of $g$ at the point $(e,t)$ $$dg_{(e,t)}: \;\mathcal{T}_{(e,t)}^{\text{Zar}}\mathcal{E}=\mathcal{T}_e\underline{E}\oplus \mathcal{T}_t^{\text{Zar}} T \rightarrow  \mathcal{T}_{g.(e,t)}^{\text{Zar}}\mathcal{E}=\mathcal{T}_{g.e}\underline{E}\oplus \mathcal{T}_{gt}^{\text{Zar}}T $$ is $\mathbb{C}$-linear, where for each complex space $S$ and for each point $s \in S$, $\mathcal{T}_s^{\text{Zar}}S$ denotes the Zariski tangent space of $S$ at $s$. On one hand, $dg_{(e,t)}=(dg_e,dg_t)$ is diagonal. On the other hand, the $G$ acts on $T$ by biholomorphisms. Therefore, it is reduced to checking that $$dg_e:\;(\mathcal{T}_e\underline{E},J_{e,t}) \rightarrow (\mathcal{T}_{ge}\underline{E},J_{ge,gt})$$ is $\mathbb{C}$-linear. However, this follows immediately from (\ref{e4.5}), the fact that $g$ is biholomorphic on the central fiber and \cite[Lemma 3.1]{3}. In brief, we have just defined a compatible $G$-action on $\mathcal{E}$ by biholomorphisms, satisfying all the conditions given in Definition \ref{d3.4}.
\end{proof}
\begin{rem} A local chart of $Q$ is given by the harmonic projection $$\mathrm{P}_{\mathcal{H}^{0,1}}:\;  Q \rightarrow \mathcal{H}^{0,1}$$ whose target can be identified with $\mathbb{C}^{\dim_{\mathbb{C}} \mathcal{H}^{0,1}}$. This map turns out to be the restriction of the usual Kuranishi map on $Q$
\begin{align*}
\mathcal{K} :Q \subset A^{0,1}(\End(E))_k  &\rightarrow A^{0,1}(\End(E))_k\\
\alpha &\mapsto \alpha -\frac{1}{2}D_E^* \mathcal{G}[\alpha,\alpha],
\end{align*} by the definition of $\mathcal{P}_{D_E}$ and  by the decomposition (\ref{e4.2}).

\end{rem}
\begin{coro}\label{c4.1}Let $X$ be a compact complex manifold over which a holomorphic vector bundle $E$ is defined. Let $G$ be a complex reductive Lie subgroup of the automorphism group of $E$. Then there exists a holomorphic $G$-equivariant semi-universal local deformation of $E$ where the extended holomorphic $G$-actions are local. 

\end{coro}
\begin{proof} Let $K$ be the connected real maximal compact subgroup of $G$ such that its complexification is $G$. We repeat the proof of Theorem \ref{t4.1} to obtain a $K$-equivariant analytic maps of germs of (Banach) analytic spaces $$\omega:\; (T,0) \rightarrow (\mathcal{P}_{D_E}^{-1}(0),0).$$ By \cite[Theorem 5.1]{3}, we can equip local $G$-actions on $(T,0)$ and on $(\mathcal{P}_{D_E}^{-1}(0),0)$ (extending the initial $K$-actions), with respect to which the map $\omega$ is $G$-equivariant. Finally, we can use the same argument as in the proof of Theorem \ref{t4.1} to construct a holomorphic $G$-equivariant semi-universal local deformation $\mathcal{E}$ of $E$ where the extended holomorphic $G$-actions are local.
\end{proof}
\begin{rem} The extended $G$-actions constructed in the proof of Theorem \ref{t4.1} are global while those in that of Corollary \ref{c4.1} are only local.
\end{rem}
Corollary \ref{c4.1} tells us that if the automorphism group $\Au(E)$ of a holomorphic vector bundle $E$ is reductive then $\Au(E)$ acts holomorphically on the base $(T,0)$ of its semi-universal local deformation. A somehow natural question arising along this story is to describe the local structure of the moduli space in terms of Kuranishi spaces. More precisely, if we think of $E$ as a point in the ``moduli space" $\mathcal{M}(\underline{E})$ of holomorphic complex bundle structures on $\underline{E}$, it would be interesting to know whether a neighborhood of $E$ in $\mathcal{M}(\underline{E})$ can be modeled on the quotient $T/\Au(E)$ in some sense. We refer the curious reader to the papers \cite{16}, \cite{17}  of F. Catanese and \cite{18} of L. Meerssemann for an analogue discussion where Teichm\"{u}ller space and Kuranishi space of complex structures on a given differentiable manifold are taken into account.

\section{The associated differential graded Lie algebra} \label{s5} Over a field of characteristic zero, a well-known theorem of Lurie in \cite{11} (obtained independently by J. Pridham in \cite{20}) claims that any reasonable moduli problem is controlled by a differential graded Lie algebra (dgLa). The philosophy hidden behind this theorem is often credited to many big names in the domain: P. Deligne and V. Drinfeld first and foremost,
then M. Kontsevich, J. Stasheff, M. Schlessinger, S. Barannikov, V. Schechtman, V. Hinich, M. Manetti. In this approach, given an object $X$ of which one wishes to study small variations (complex compact manifolds, algebraic schemes, vector bundles, isolated singularities, etc), the philosophy suggests that there exists a dgLa $\mathfrak{g}_X$ governing deformations of $X$ in the sense that the deformation functor which to each local Artin algebra $A$ associates the set of Maurer-Cartan solutions modulo the gauge action (defined by means of $\mathfrak{g}_X$), is isomorphic to the set of isomorphism classes of deformations of $X$ over $\Sp(A)$. One of the classic illustrations of this phenomenon is when $X$ is a complex compact manifold. In this case, the controlling dgLa is nothing but the Dolbeault complex with values in the holomorphic tangent bundle of $X$. This allows us to transform a purely geometric problem to a purely algebraic one in view that the associated dgLa gives almost all information about the initial local moduli problem: its $0^{\text{th}}$, $1^{\text{st}}$ and $2^{\text{nd}}$ cohomology groups are nothing but the space of infinitesimal automorphisms, that of first order infinitesimal deformations (or equivalently, the tangent space) and that of obstructions to the (formal) smoothness, respectively. For more historical details about this direction, the interested reader is referred to the exceptionally beautiful seminar paper \cite{19} of B. To\"{e}n.

Continuing this spirit, in this section, we first translate the deformation problem of vector bundle in $\S$\ref{s3} into the language of functors of artinian rings and then compute the dgLa associated to the local moduli problem of holomorphic vector bundle. Hence, a formal version of Corollary \ref{c4.1} follows immediately from the mechanism that we developed in \cite{5}. Let us first recall some standard conventions:
\begin{enumerate}
\item $\set$ is the category of sets.
\item $\grp$ is the category of groups.
\item $\art$ is the category of local artinian $\mathbb{C}$-algebras with residue field $\mathbb{C}$. For each $A \in \art$, we denote its associated germ of complex spaces and its maximal ideal by $\mathrm{Spec}(A)$ and $\mathfrak{m}_A$, respectively.
\end{enumerate} 
In the language of artinian rings, Definition \ref{d3.1} can be read as follows. 
\begin{defi} \begin{enumerate} 
\item[(i)] \label{d5.1i} A deformation of $E$ over $A \in  \mathrm{\art}$ is a section $\alpha \in A^{0,1}(\End(E)) \otimes \mathfrak{m}_A$ such that \begin{equation}D_E.\alpha +\alpha \wedge \alpha =0 
\end{equation}
in $ A^{0,2}(E)\otimes\mathfrak{m}_A$.
\item[(ii)]\label{d5.1ii} Two deformations $\alpha_1, \alpha_2$ of $E$ over $A$ are equivariant if there exists a section $$\rho \in A^{0,0}(\GL(E))\otimes \mathfrak{m}_A $$ inducing the identity section on $E$ such that $$\rho^{-1}\circ(D_E+\alpha_2)\circ \rho=D_E+\alpha_1.$$ 

\end{enumerate}
\end{defi}
Now, consider the Dolbeault complex with values in the endomorphism bundle $\End(E)$ of $E$ $$A^{0,0}(\End(E)) \overset{D_E}{\longrightarrow} A^{0,1}(\End(E)) \overset{D_E}{\longrightarrow} A^{0,2}(\End(E))\overset{D_E}{\longrightarrow}\cdots\overset{D_E}{\longrightarrow}A^{0,n}(\End(E))$$ which can be further equipped with a graded Lie structure by using the following Lie bracket
\begin{equation}\label{e5.2}[\phi d\bar{z}_I, \psi  d\bar{z}_J]=(\phi\circ\psi-(-1)^{\left | I \right |.\left | J \right |} \psi\circ \phi) d\bar{z}_I \wedge\bar{z}_J \end{equation} where $n:=\dim{X}$, $I,J\subset \lbrace 1,\ldots, n \rbrace$ and $z_1,\ldots, z_n$ are local holomorphic coordinates. We denote this dgLa by $\mathfrak{g}_{*}$. Observe that the relation (\ref{d5.1i}) becomes $$D_E\alpha+\frac{1}{2}[\alpha,\alpha]=0$$ which is in the form of a Maurer-Cartan equation. As a matter of fact, the functor of artinian rings corresponding to the local moduli problem of $E$ is given by
\begin{align*}
\defo_E :\art &\rightarrow \set \\
A &\mapsto \left \{\alpha \in A^{0,1}(\End(E))\otimes \mathfrak{m}_A  \mid D_E\alpha+\frac{1}{2}[\alpha,\alpha]=0  \right \}/\sim 
\end{align*} where the equivalence relation $\sim$ is given in Definition \ref{d5.1ii}(ii). 

Finally, for the completeness, we recall the classical deformation functor $\mathbf{\mathrm{MC}}_{\mathfrak{g}_*}$ associated to $\mathfrak{g}_*$, defined via the Maurer-Cartan equation. We have two functors:
\begin{itemize}
\item[(1)] The Gauge functor 
\begin{align*}
G_{\mathfrak{g}_*}:\; \art & \rightarrow \grp \\
A &\mapsto \mathrm{exp}(\mathfrak{g}_0\otimes \mathfrak{m}_A)
\end{align*}

\item[(2)] The Maurer-Cartan functor $MC_{\mathfrak{g}_*}:\; \art \rightarrow \set $ defined by
\begin{align*}
MC_{\mathfrak{g}_*}:\; \art & \rightarrow \grp \\
A & \mapsto \left \{ x \in \mathfrak{g}_1\otimes \mathfrak{m}_A\mid D_E x+\frac{1}{2}[x,x]=0  \right \}.
\end{align*} 
\end{itemize}
For each $A$, the gauge action of $G_{\mathfrak{g}_*}(A)$ on the set $MC_{\mathfrak{g}_*}(A)$ is functorial in $A$ and gives an action of the group functor  $G_{\mathfrak{g}_*}$ on $MC_{\mathfrak{g}_*}$. This allows us to define the quotient functor \begin{align*}
\mathbf{\mathrm{MC}}_{\mathfrak{g}_*}:\; \art & \rightarrow \set \\
A & \mapsto MC_{\mathfrak{g}_*}(A)/G_{\mathfrak{g}_*}(A),
\end{align*} 
\begin{thm} As a consequence, there is an isomorphism
$$\mathbf{\mathrm{MC}}_{\mathfrak{g}_*}\cong \defo_E$$ as functors of artinian rings. As a sequence, the differential graded Lie algebra controlling the deformations of $E$ is $$A^{0,0}(\End(E)) \overset{D_E}{\longrightarrow} A^{0,1}(\End(E)) \overset{D_E}{\longrightarrow} A^{0,2}(\End(E))\overset{D_E}{\longrightarrow}\cdots\overset{D_E}{\longrightarrow}A^{0,n}(\End(E))$$ where the differential is given by the connection $D_E$ and the Lie bracket given by the rule (\ref{e5.2}).
\end{thm} 
\begin{proof}
The local isomorphism 
$$ \mathrm{exp}:\;(A^{0,0}(\End(E)),0) \rightarrow (A^{0,0}(\GL(E)),\mathrm{Id}_E) $$ and the fact that $(A^{0,0}(\GL(E)),\mathrm{Id}_E)$ acts on $A^{0,1}(\End(E))$ by conjugations permit us to conclude that the equivalence relation $\sim$ is given in Definition \ref{d5.1ii}(ii) is the same as the one induced by the gauge action of $G_{\mathfrak{g}_*}(A)$. Therefore, the desired isomorphism follows immediately.
\end{proof}

\begin{coro} \label{c5.1} Let $X$ be a compact complex manifold over which a holomorphic vector bundle $E$ is defined. Let $G$ be a complex reductive Lie subgroup of the automorphism group of $E$. Then there exists a compatible formal $G$-action on the local moduli space of $E$. 
\end{coro}
\begin{proof}
The  functor $\mathbf{\mathrm{MC}}_{\mathfrak{g}_*}$ can be naturally upgraded to a derived formal moduli problem $F_{\mathfrak{g}_*}$  in Lurie's sense (cf. \cite{11}) via a simplicial version of the Maurer-Cartan equation (see \cite{8} for such a construction). Moreover, the associated dgLa of $F_{\mathfrak{g}_*}$  is nothing but $\mathfrak{g}_*$. Consequently, $F_{\mathfrak{g}_*}$ is a naturally extension of $\defo_E$ in the derived world. 

Now, note that the action of $G$ on $E$ induces a natural a $G$-action on $E$. By the same argument as in \cite[Lemma 3.1]{5}, we can write $\mathfrak{g}_*$ as a homotopic colimit of ``simple" dgLas, i.e.
$$\mathfrak{g}_*=\mathrm{colim}_i \; \mathfrak{g}(i)_*$$ where \begin{enumerate}
\item[(i)] each $\mathfrak{g}(i)_k$ is finite-dimensional,
\item[(ii)] $\mathfrak{g}(i)_*$ is cohomologically concentrated in $[0,+\infty)$,
\item[(iii)] each $ \mathfrak{g}(i)_*$ carries a $G$-action and the colimit of these $G$-actions gives back the initial $G$-action on $\mathfrak{g}_*$. 
\end{enumerate}  A remark is in order. Even in the formal aspect, to make the above $G$-approximation of the associated dgLa $\mathfrak{g}_*$ possible, the $G$-equivariant Hodge decomposition $$A^{0,n}(\End(E)) =\mathcal{H}^{0,n}\oplus \square_{E} A^{0,n}(\End(E)),$$ which is purely analytic still plays a crucial role.

By \cite[Theorem 2.3]{5}, the semi-prorepresentable object of $F_{\mathfrak{g}_*}$ carries a $G$-action. Hence, the restriction of $F_{\mathfrak{g}_*}$ on $\art$, (which is nothing but $\defo_E$) has a semi-universal element whose base is equipped with a compatible $G$-action. This finishes the proof. 
\end{proof}
\begin{rem} Corollary \ref{c5.1} reflects the fact that for deformation problems, a formal solution is somehow easy to produce whereas Corollary \ref{c4.1} tells us that among formal solutions, we can extract a convergent one.  
\end{rem}
\section{Perspectives}\label{s6}

In this final section, we summarize what we did in this paper and in \cite{3} in a more general setting in terms of associated dgLas (we also refer the reader to \cite{15} for a version without group actions). 

To start, we consider the deformation problem of an analytic object $X_0$, whose associated controlling differential graded Lie algebra is $(\mathfrak{g}_*,d)$. As usual, the space of infinitesimal deformations and that of obstructions are the first and the second cohomology of $\mathfrak{g}_*$, i.e. $H^1(\mathfrak{g}_*)$ and  $H^2(\mathfrak{g}_*)$, respectively. Let \begin{align*}
MC_{\mathfrak{g}_*} :\mathfrak{g}_1 &\rightarrow \mathfrak{g}_2\\
\alpha &\mapsto d\alpha +\frac{1}{2}[ \alpha, \alpha]
\end{align*} be the Maurer-Cartan equation associated to $\mathfrak{g}_*$. Any subgroup $G$ of the automorphism group of $X_0$ induces a natural $G$-action on each component of $\mathfrak{g}_*$ compatible with the differential $d$. We assume further that there are good analytic structures on $\mathfrak{g}_0$,  $\mathfrak{g}_1$ and $\mathfrak{g}_2$ where the implicit function theorem is available (for example, Banach analytic spaces) and there exists a $G$-invariant metric on $\mathfrak{g}_*$, with respect to which we are able to compute the formal adjoint $d^*$ of degree $-1$. Let us denote $\square := dd^*+d^*d$. Supposedly, we have a decomposition 
\begin{align*}
\mathfrak{g}_1 &=\ker\square \bigoplus \mathrm{Im}\square 
\end{align*} together with two linear operators:
\begin{enumerate}
\item[(i)] The ``Green operator": $\mathcal{G}:\; \mathfrak{g}_1 \rightarrow \mathrm{Im}\square $,
\item[(ii)] The ``harmonic projection": $\mathrm{P}_{\ker\square}:\;\mathfrak{g}_1  \rightarrow \ker\square$
\end{enumerate} such that $$\mathrm{Id}_{\mathfrak{g}_1}= \mathrm{P}_{\ker\square}+\square\mathcal{G}$$ and $\ker\square $ can be naturally identified with $H^1(\mathfrak{g}_*)$. Consider the following ``Kuranishi map"
\begin{align*}
\mathcal{K} :\mathfrak{g}_1 &\rightarrow \mathfrak{g}_1\\
\alpha &\mapsto\alpha +\frac{1}{2}d^*\mathcal{G} [\alpha, \alpha].
\end{align*}
\begin{thm} There exists a compatible $G$-action on the local moduli space of $X_0$.
\end{thm}
\begin{proof} Let us denote by $N$ the following space 
$$\lbrace \alpha \in  \mathfrak{g}_1 \mid \left( \mathcal{K}-\mathrm{P}_{\ker\square}\right)(\alpha)=0 \rbrace.$$
Then it can be checked that the germ of analytic space $$(T,0):=(N,0)\cap(MC_{\mathfrak{g}}^{-1}(0),0)$$ is the desired ``Kuranishi space" (see \cite[Theorem 3.1]{15} for such a verification). The existence of group operations on $(T,0)$ follows immediately from the $G$-equivariance of all the maps and of all the operators involved. 
\end{proof}
\begin{rem} The key point here is the existence of the $G$-invariant metric and that of the splitting $$\mathrm{Id}_{\mathfrak{g}_1}= \mathrm{P}_{\ker\square}+\square\mathcal{G}.$$ The former is assured if $G$ is a compact Lie group by the unitary trick while the latter can come from the Hodge theory if we deal with complex compact manifolds. In general, we do not have such powerful tools. 
\end{rem}
The existence of reductive group operations on the Kuranishi space of complex compact manifolds (cf. \cite{3}) and that on the Kuranishi space of holomorphic vector bundles, dealt in this paper, can be thought of as living illustrations of the following philosophy. 

\textit{\enquote{Reductive subgroups of the automorphism group of the analytic object under deformation can be (at least locally) analytically extended to its semi-universal deformation.}}

In other words, there should be a compatible extended action on the ``Kuranishi space", which permutes nearby complex structures and the initial group action might be regarded as the stabilizer group with respect to the prescribed complex structure (corresponding to the reference point).  The formal aspect of this philosophy was systematically treated in the groundbreaking work of D. S. Rim, as mentioned in the introduction, in which a formal extendability of reductive actions is guaranteed, unique up to non-canonical equivariant isomorphisms, for any homogeneous fibred category in groupoid. However, the convergence of his construction, which is necessarily required in the analytic setting, is extremely hard to prove even in simple cases. Therefore, analytically speaking, a rigorously mathematical formulation of this philosophy might be a good problem to work on. 

\bibliographystyle{amsplain}
{}

\end{document}